\documentclass[leqno]{article}
\usepackage{amsmath,amsthm,amssymb,verbatim}

\newtheorem{theorem}{Theorem}[section]

\newtheorem{lemma}[theorem]{Lemma}

\newtheorem{proposition}[theorem]{Proposition}
\theoremstyle{remark}
\newtheorem{remark}{Remark}[section]

\begin{document}

\title{\textbf{First-order equivalent to
Einstein-Hilbert Lagrangian}}

\author{\textsc{M. Castrill\'on L\'opez}\\
Instituto de Ciencias Matem\'aticas CSIC-UAM-UC3M-UCM\\
Departamento de Geometr\'{\i}a y Topolog\'{\i}a\\
Facultad de Matem\'aticas, UCM\\Avda.\ Complutense s/n,
28040-Madrid, Spain\\
\emph{E-mail:\/} \texttt{mcastri@mat.ucm.es}
\and \textsc{J. Mu\~{n}oz Masqu\'e}\\
Instituto de F\'{\i}sica Aplicada, CSIC\\
C/ Serrano 144, 28006-Madrid, Spain\\
\emph{E-mail:\/} \texttt{jaime@iec.csic.es}
\and \textsc{E. Rosado Mar\'{\i}a}\\
Departamento de Matem\'atica Aplicada\\Escuela T\'ecnica
Superior de Arquitectura, UPM\\
Avda.\ Juan de Herrera 4, 28040-Madrid, Spain\\
\emph{E-mail:\/} \texttt{eugenia.rosado@upm.es}}
\date{}

\maketitle

\begin{abstract}
\noindent
A first-order Lagrangian $L^\nabla $ variationally equivalent
to the second-order Einstein-Hilbert Lagrangian is introduced.
Such a Lagrangian depends on a symmetric linear connection,
but the dependence is covariant under diffeomorphisms.
The variational problem defined by $L^\nabla $ is proved
to be regular and its Hamiltonian formulation is studied,
including its covariant Hamiltonian attached to $\nabla $.

\end{abstract}

\noindent\emph{PACS 2010:\/} 04.20.Cv, 04.20.Fy, 04.50.-h,
11.10.Kk.

\noindent\emph{Mathematics Subject Classification 2010:\/}
58A20, 58E11, 58E30, 83C05.

\noindent\emph{Key words:} Bundle of metrics, Einstein-Hilbert
Lagrangian, jet bundle, linear connection.

\noindent\emph{Acknowledgements:\/} Supported by Ministerio
de Ciencia e Innovaci\'on of Spain under grant \#MTM2011--22528.

\section{Introduction}
Let $p\colon \mathfrak{M}\to  M$ be the bundle of pseudo-Riemannian
metrics of a given signature $(n^+,n^-)$, $n^++n^-=n=\dim M$,
over a connected $C^\in fty $ manifold oriented by a volume form
$\mathbf{v}\in \Omega ^n(M)$. The Einstein-Hilbert (or E-H for short)
functional is the second-order Lagrangian density $L_{EH}\mathbf{v}$
on $\mathfrak{M}$ defined along a metric $g$ by $s^g\mathbf{v}_g$,
where $s^g$\ denotes the scalar curvature of $g$ and $\mathbf{v}_g$
its Riemannian volume form; namely,
\begin{equation}
L_{EH}\circ  j^2g
\! = \! \sqrt{|\det(g_{ab})|}g^{jk}
\! \left\{
\!
\tfrac{\partial
\left(
\Gamma ^g
\right) _{jk}^i}
{\partial x^i}
\! - \! \tfrac{\partial
\left(
\Gamma ^g
\right)  _{ik}^i}
{\partial x^j}
\!+\! \left(
\Gamma ^g
\right) _{jk}^l
\left(
\Gamma ^g
\right) _{il}^i
\! - \! \left(
\Gamma ^g
\right) _{ik}^l
\left(
\Gamma ^g
\right) _{jl}^i
\!
\right\} ,
\label{LEHbis}
\end{equation}
where $
\left(
\Gamma ^g
\right) _{jk}^i$ are the Christoffel symbols
of the Levi-Civita connection $\nabla ^g$ of the metric $g$.
As is known (e.g., see \cite[3.3.1--3.3.2]{Carmeli}), the first-order
Lagrangian $L_{1}$ defined along $g$ by
$\sqrt{|\det (g_{ab})|}g^{jk}
((\Gamma ^g)_{ij}^l(\Gamma ^g)_{kl}^i
-(\Gamma ^g)_{jk}^l(\Gamma ^g)_{il}^i)$
differs from $L_{EH}$ by a divergence term, but unfortunately $L_1$
is \emph{not} an invariantly defined quantity.

Below, a geometrically defined first-order Lagrangian $L^\nabla $
(depending on an auxiliary symmetric linear connection $\nabla $ on $M$)
is introduced, which is variationally equivalent to E-H Lagrangian $L_{EH}$
and, consequently, it has the same Euler-Lagrange equations, namely Einstein's
field equations in the vacuum for arbitrary signature. In particular, this
explains why the E-H Lagrangian admits a true first-order Hamiltonian formalism.

Although $L^\nabla $ depends on an auxiliary symmetric linear connection,
this dependence is natural with respect to the action of diffeomorphisms
of $M$ on connections and on Lagrangian functions, as proved in section
\ref{dependence} below. This fact justifies the construction of such
a Lagrangian and the interest of its existence.

Furthermore, the Lagrangian $L^\nabla $ is seen to be regular and its
Hamiltonian formulation is studied, computing explicitly its momenta
functions and the covariant Hamiltonian attached to $\nabla $
in the sense of \cite{MSh}.
\section{The equivalent Lagrangian $L^\nabla $\ defined}
The difference tensor field between the Levi-Civita connection $\nabla ^g$
of a metric $g$ and a given symmetric linear connection $\nabla $ on $M$
is the $2$-covariant $1$-contravariant tensor given by,
\[
T^{g,\nabla }=\nabla ^g-\nabla
=\left(
(\Gamma ^g)_{ij}^h-\Gamma _{ij}^h
\right)
dx^i\otimes dx^j\otimes
\frac{\partial}{\partial x^h},
\]
where $(\Gamma ^g)_{jk}^i$ (resp.\ $\Gamma _{jk}^i$) are the Christoffel
symbols of the connection $\nabla ^g$ (resp.\ $\nabla $). A Lagrangian
function $L^\nabla $ on the bundle of metrics
$p\colon \mathfrak{M}\to M$ is defined as follows:
\begin{equation}
L^\nabla
\left(
j_x^2g
\right)
\mathbf{v}_x
=\left\{
s^g(x)
+c\left(
\left(
\operatorname*{alt}\nolimits_{23}
\left(
\nabla^gT^{g,\nabla
}\right)  _x
\right) ^\sharp
\right)
\right\}
\left(
\mathbf{v}
_g\right) _x,
\quad
\forall  j_x^2g\in J^2\mathfrak{M},
\label{L'}
\end{equation}
where we confine ourselves to consider coordinate systems
$(x^1,\dotsc,x^n)$ on $M$ adapted to $\mathbf{v}$, i.e.,
\[
\mathbf{v}=dx^1\wedge \ldots\wedge  dx^n,
\qquad
\mathbf{v}_g
=\sqrt{
\left\vert
\det
\left(
g_{uv}
\right)
\right\vert }
\mathbf{v},
\qquad
g=g_{uv}dx^u\otimes dx^v,
\]
$\operatorname*{alt}\nolimits_{23}
\colon \otimes ^3T^\ast M\otimes TM
\to \otimes ^3T^\ast M\otimes TM$ denotes the alternation
of the second and third covariant indices,
$^\sharp \colon \otimes ^3T^\ast M\otimes TM
\to \otimes^2T^\ast M\otimes^2TM$ is the isomorphism induced by $g$,
\[
w_1\otimes w_2\otimes w_3\otimes X\mapsto w_1\otimes w_2
\otimes (w_3)^\sharp \otimes X,
\quad
\forall  X\in T_xM,
\;
\forall w_1,w_2,w_3\in T_x^\ast M,
\]
and, finally, $c\colon \otimes^2T^\ast M\otimes^2TM\to \mathbb{R}$
denotes the (total) contraction of the first and second covariant indices with
the first and second contravariant ones, respectively. We write $L^\nabla $
in order to emphasize the fact that the Lagrangian depends on the auxiliary
symmetric linear connection $\nabla $ previously chosen.

If $y_{ij}=y_{ji}$, $i,j=1,\dotsc,n$, are the coordinates on the fibres
of $p$ induced from a coordinate system $(x^h)_{h=1}^n$ on $M$, namely,
$g_x=y_{ij}(g_x)dx^i\otimes dx^j$ for every metric $g_x$ over $x\in
M$, and $(x^h,y_{ij},y_{ij,k},y_{ij,kl}=y_{ij,lk})$ denotes the coordinate
system induced on $J^2\mathfrak{M}$, then $L_{EH}$ is locally given by,
\begin{equation}
L_{EH}=\rho
\left(
y^{ac}y^{bd}-y^{ab}y^{cd}
\right)
y_{ab,cd}+L_0,
\label{LEH}
\end{equation}
where
\begin{equation}
\left\{
\begin{array}
[c]{l}
L_0=\rho y^{ij}
\left\{
y^{hm}
\left(
y_{mr,j}G_{ih}^r-y_{mr,h}G_{ij}^r
\right)
+G_{ij}^mG_{hm}^h-G_{ih}^mG_{jm}^h
\right\} ,
\smallskip\\
\rho=\sqrt{
\left\vert
\det
\left(
y_{ij}
\right)
\right\vert },
\end{array}
\right.
\label{L0_rho}
\end{equation}
and $G_{rj}^i\colon  J^1\mathfrak{M}\to \mathbb{R}$
are defined by,
$G_{rj}^i=\tfrac{1}{2}y^{is}
\left(
y_{rs,j}+y_{js,r}-y_{rj,s}
\right) $.

If $L^{\prime \nabla }$ is the second-order Lagrangian on $\mathfrak{M}$
determined by the second summand of the right-hand side in the formula
\eqref{L'}, namely
\[
L^{\prime \nabla }(j_x^2g)=c
\left(
\left(
\operatorname*{alt}
\nolimits_{23}
\left(
\nabla^gT^{g,\nabla }
\right) _x
\right) ^\sharp
\right) ,
\]
then \eqref{L'} can equivalently be rewritten as follows:
$L^\nabla =L_{EH}+\rho L^{\prime \nabla }$ and as a calculation shows,
\begin{align}
L^{\prime \nabla }\circ  j^2g
& =g^{jr}
\left\{
\tfrac{\partial
\left(
T^{g,\nabla }
\right)  _{ri}^i}
{\partial x^j}-\tfrac{\partial
\left(
T^{g,\nabla }
\right)  _{rj}^i}
{\partial x^i}
\right.
\label{Summand}\\
& +\left(
\Gamma ^g
\right)  _{ji}^a
\left(
T^{g,\nabla }
\right) _{ra}^i
-\left(
\Gamma ^g
\right) _{jr}^a
\left(
T^{g,\nabla }
\right) _{ai}^i
\nonumber\\
& \left.
+\left(
\Gamma ^g
\right)  _{ir}^a
\left(
T^{g,\nabla }
\right) _{aj}^i
-\left(
\Gamma ^g
\right)  _{ai}^a
\left(
T^{g,\nabla }
\right)
_{rj}^i\!
\right\}
.\nonumber
\end{align}
\begin{lemma}
\label{lemma1}
The Lagrangian $L^\nabla $ is of first order.
\end{lemma}
\begin{proof}
Taking the definition of $T^{g,\nabla }$ and the formulas
\eqref{Summand}, \eqref{LEHbis} into account, one obtains
{\small
\begin{align}
\sqrt{
\left\vert
\det
\left(
g_{uv}
\right)
\right\vert }
(L^{\prime\nabla }\circ  j^2g)
& =-L_{EH}\circ  j^2g
\label{Summand2}\\
& +\sqrt{
\left\vert
\det\left(
g_{uv}
\right)
\right\vert }
g^{jr}
\left\{
\left(
\Gamma ^g
\right) _{ji}^a
\left(
\Gamma ^g
\right) _{ra}
^i-\left(
\Gamma ^g
\right) _{ai}^a
\left(
\Gamma ^g
\right) _{rj}
^i\right\}
\nonumber\\
& -\sqrt{
\left\vert
\det
\left(
g_{uv}
\right)
\right\vert }
g^{jr}
\left\{
\dfrac{\partial\Gamma _{ri}^i}{\partial x^j}
-\dfrac{\partial\Gamma _{rj}^i}{\partial x^i}
+\left(
\Gamma ^g
\right)  _{ji}^a\Gamma _{ra}
^i\right.
\nonumber\\
& \left.
-\left(
\Gamma ^g
\right)  _{jr}^a
\Gamma _{ai}^i
+\left(
\Gamma ^g
\right) _{ir}^a
\Gamma _{aj}^i
-\left(
\Gamma ^g
\right)
_{ai}^a\Gamma _{rj}^i
\right\} .
\nonumber
\end{align}
}
Hence $(\rho L^{\prime \nabla }+L_{EH})\circ  j^2g$
depends on the values of the metric $g$
and its first derivatives only.
\end{proof}
In fact, the following local expression is readily deduced:
\[
L^\nabla =\rho y^{jr}
\left\{
G_{ji}^aT_{ra}^i-G_{ai}^aT_{rj}^i
+G_{jr}^a\Gamma _{ai}^i
-G_{ir}^a\Gamma _{aj}^i
-\frac{\partial \Gamma _{ri}^i}{\partial x^j}
+\frac{\partial\Gamma _{rj}^i}{\partial x^i}
\right\} ,
\]
$T_{jk}^i\colon  J^1\mathfrak{M}\to \mathbb{R}$ being the functions
defined by, $T_{ij}^h=G_{ij}^h-\Gamma _{ij}^h$.
\begin{remark}
\label{remark1}
As $L^\nabla $ has a geometrical definition, the local
expression above actually provides a global Lagrangian.
Moreover, if $\nabla $ is a flat linear connection
and one considers an adapted coordinate system to $\nabla $
(i.e., a coordinate system on which all the Christoffel symbols
of $\nabla $ vanish), then the local expression for $L^\nabla $
coincides with the local Lagrangian $L_1$ defined
in the introductory section.
\end{remark}
\section{$L^\nabla $ and $L_{EH}$ are variationally equivalent}
As a computation shows, the second summand in the definition
of $L^\nabla $ can be rewritten in terms of the metric $g$
and the auxiliary connection $\nabla $ only, as follows:
\begin{align*}
c\left(
\left(
\operatorname*{alt}\nolimits_{23}
\left(  \
nabla
^gT^{g,\nabla }
\right)
\right) ^\sharp
\right)
& =\left(
g^{js}g^{ir}-g^{jr}g^{is}
\right)
g_{ri,sj}\\
& +\tfrac{1}{2}
\left\{
\left(
2g^{ir}g^{jb}-g^{bi}g^{rj}-g^{br}g^{ij}
\right)
g^{as}
\right. \\
& +\left(
g^{ar}g^{ib}+g^{bi}g^{ra}-2g^{ir}g^{ab}
\right)
g^{js}\\
& -\left(
g^{sr}g^{jb}-g^{br}g^{sj}
\right)
g^{ai}\\
& \left.
-\left(
g^{ar}g^{sb}-g^{sr}g^{ab}
\right)
g^{ij}
\right\}
g_{ab,j}g_{rs,i}\\
& -g^{jr}
\left(
\tfrac{\partial\Gamma _{ri}^i}{\partial x^j}
-\tfrac{\partial\Gamma _{rj}^i}{\partial x^i}
\right) \\
& +\tfrac{1}{2}
\left\{
\left(
2g^{js}g^{ar}-g^{jr}g^{as}
\right)
g_{rj,s}\Gamma _{ai}^i\right. \\
& \left.
+\left(
g^{jr}g^{ab}-2g^{ar}g^{jb}
\right)
g_{ab,i}\Gamma _{rj}^i
\right\}  .
\end{align*}
\begin{lemma}
\label{lemma2}
If $D_i$ denotes the total derivative with respect to $x^i$,
then
\[
c\left(
\left(
\operatorname*{alt}\nolimits_{23}
\left(
\nabla
^gT^{g,\nabla }
\right)
\right)  ^\sharp
\right)
\mathbf{v}_g
=-\left(
D_i((L_{EH})_{\nabla}^i)\circ  j^2g
\right)
\mathbf{v},
\]
where
\begin{equation}
\left(
L_{EH}
\right)  _\nabla ^i
=\sum_{c\leq r}\tfrac{1}{2-\delta_{ib}}
\frac{\partial L_{EH}}{\partial y_{cr,ib}}
\left(
y_{cr,b}
-\left(
\Gamma _{bc}^ay_{ar}+\Gamma _{br}^ay_{ac}
\right)
\right) .
\label{Li}
\end{equation}
\end{lemma}
From this lemma it follows that $L^\nabla $ and $L_{EH}$
are variationally equivalent as, according to the formula
\eqref{L'}, one has
\begin{align*}
\left(
L^\nabla \circ  j^2g
\right)
\mathbf{v}
& =\left(
L_{EH}\circ j^2g
\right)
\mathbf{v}
+c\left(
\left(
\operatorname*{alt}\nolimits_{23}
\left(
\nabla^gT^{g,\nabla }
\right)
\right) ^\sharp
\right)
\mathbf{v}_g\\
& =\left\{
\left(
L_{EH}-D_i
\left(
\left(
L_{EH}\right) _\nabla ^i
\right)
\right)
\circ  j^2g
\right\}
\mathbf{v}.
\end{align*}
Hence $L^\nabla =L_{EH}-D_i
\left(
\left(
L_{EH}
\right) _{\nabla}^i
\right) $
and therefore, $L^\nabla $ and $L_{EH}$ differ in a total divergence.

The proof of Lemma \ref{lemma2} follows by computing $D_i
\left(
\left(
L_{EH}
\right) _\nabla ^i
\right) $ using \eqref{LEH} and \eqref{Li}, taking the identity
$D_i\rho =\tfrac{\rho }{2}y^{rs}y_{rs,i}$ into account,
after a simple---but rather long---computation.
\section{Dependence on $\nabla $\label{dependence}}
Below, the dependence of the Lagrangian $L^\nabla $ with respect
to the symmetric linear connection $\nabla $, is analysed. First,
some geometric preliminaries are introduced.

The image of a linear connection $\nabla $ by a diffeomorphism
$\phi\colon M\to  M$ is defined to be $\left(
\phi \cdot \nabla \right) _XY=\phi \cdot
\left(
\nabla _{\phi ^{-1}\cdot X}(\phi ^{-1}\cdot Y)
\right) $,
$\forall  X,Y\in \mathfrak{X}(M)$. As is well known (e.g.,
see \cite[p.\ 643]{Epstein}), the Levi-Civita connection
of a metric transforms according to the rule:
$\phi ^{-1}\cdot \nabla ^g=\nabla ^{\phi ^\ast g}$.
Hence the following formulas hold:
\[
\phi ^{-1}\cdot T^{g,\nabla }
=T^{\phi ^\ast g,\phi ^{-1}\cdot \nabla },
\quad
S^{\phi \cdot \nabla }
=(\phi ^{-1})^\ast S^\nabla
=\phi \cdot S^\nabla ,
\quad
s^g=s^{\phi ^\ast g},
\]
where$S^\nabla (X,Y)=\operatorname*{trace}(Z\mapsto R^\nabla (Z,X)Y)$
is the Ricci tensor of $\nabla $ (e.g., see \cite[VI, p.\ 248]{KN}).
Moreover, the lift of $\phi $ to the bundle of metrics
$p\colon \mathfrak{M}\to  M$ is given by
$\bar\phi (g_x)=(\phi ^{-1})^\ast g_x$, $\forall  g_x\in p^{-1}(x)$
(cf.\ \cite{MR}); hence $p\circ \bar{\phi }=\phi \circ  p$,
and the mapping $\bar{\phi }\colon \mathfrak{M}\to \mathfrak{M}$
has an extension to the $r$-jet bundle
$\bar{\phi }^{(r)}\colon J^r\mathfrak{M}\to J^r\mathfrak{M}$
defined by, $\bar\phi ^{(r)}
\left(
j_x^rg
\right)
=j_{\phi(x)}^r(\bar\phi \circ  g\circ \phi ^{-1})$.

Let $\mathbf{v}_{\mathfrak{M}}$ be the nowhere-vanishing
$p$-horizontal $n$-form on $\mathfrak{M}$ defined as follows:
$\left(
\mathbf{v}_{\mathfrak{M}}
\right) _{g_x}=\mathbf{v}_{g_x}$, $\forall g_x\in \mathfrak{M}$,
where, as above, $\mathbf{v}_{g_x}$ denotes the Riemannian volume
form attached to $g_x$. Hence $\mathbf{v}_{\mathfrak{M}}
=\rho\mathbf{v}$, where $\rho$ is as in \eqref{L0_rho}.
Every $r$-th order Lagrangian density $\Lambda$ on $\mathfrak{M}$
can thus be written as $\Lambda=L\mathbf{v}_{\mathfrak{M}}$
for a certain Lagrangian function $L\in C^\infty (J^{r}\mathfrak{M})$
and $\Lambda $ is invariant under diffeomorphisms, i.e.,
$(\bar\phi ^{(r)})^\ast \Lambda =\Lambda $,
$\forall \phi\in \operatorname*{Diff}M$, if and only if $L$ is, i.e.,
$L\circ \bar\phi ^{(r)}=L$, as $(\bar\phi ^{(r)})^\ast \Lambda
=(L\circ \bar\phi ^{(r)})(\bar\phi ^\ast \mathbf{v}_{\mathfrak{M}})$
and, according to \cite[Proposition 7]{MV}, $\mathbf{v}_{\mathfrak{M}}$
is invariant under diffeomorphisms, i.e.,
$\bar\phi ^\ast \mathbf{v}_{\mathfrak{M}}=\mathbf{v}_{\mathfrak{M}}$.

The E-H Lagrangian density $L_{EH}\mathbf{v}$ is known to be invariant
under diffeomorhisms, i.e., $(\bar\phi ^{(2)})^\ast (L_{EH}\mathbf{v)}
=L_{EH}\mathbf{v}$, $\forall \phi\in \operatorname*{Diff}M$. In fact,
there exists a classical result by Hermann Weyl (\cite[Appendix II]{Weyl},
also see \cite{Heyde}, \cite{Lovelock}), according to which the only
$\operatorname*{Diff}M$-invariant Lagrangians on $J^2\mathfrak{M}$
depending linearly on the second-order coordinates $y_{ab,ij}$
are of the form $\lambda L_{EH}+\mu$, for scalars $\lambda $, $\mu $.

Therefore, transforming the equation $L^\nabla \mathbf{v}=L_{EH}
\mathbf{v}+L^{\prime \nabla }\mathbf{v}_{\mathfrak{M}}$
by a diffeomorphism $\phi $, one obtains
$(\bar\phi ^{(1)})^\ast
\left(
L^\nabla \mathbf{v}
\right)
=L_{EH}\mathbf{v}+(L^{\prime \nabla }\circ \bar\phi
^{(2)})\mathbf{v}_{\mathfrak{M}}$, and one is led to compute
$L^{\prime \nabla }\circ \bar\phi ^{(2)}$, which, by using
the formulas above,  is proved to transform according
to the following rule:
\begin{equation}
\label{L^prime^nabla}
L^{\prime \nabla }\circ \bar{\phi }^{(2)}
=L^{\prime \phi ^{-1}\cdot \nabla }.
\end{equation}
\section{Hamiltonian formalism}
\subsection{Regularity of $L^\nabla $}
\begin{proposition}
\label{proposition1}
For $\dim M=n\geq 3$, the Lagrangian $L^\nabla $
is regular, namely, the following square matrix of size
$\frac{1}{2}n^2(n+1)$ is non-singular:
\begin{equation}
\label{matrix1}
\left(
\frac{\partial p_w^{uv}}{\partial y_{ab,c}}
\right) _{a\leq b,c}^{u\leq v,w}
=\left(
\frac{\partial^2H^\nabla }{\partial y_{ab,c}
\partial y_{uv,w}}
\right)  _{a\leq b,c}^{u\leq v,w},
\end{equation}
where
\begin{equation}
\label{p's_H}
p_k^{ij}=\frac{\partial L^\nabla }{\partial y_{ij,k}},
\quad
H^\nabla
=\sum _{i\leq j}
\frac{\partial L^\nabla }{\partial y_{ij,k}}
y_{ij,k}-L^\nabla .
\end{equation}
\end{proposition}
\begin{proof}
From the very definition of $H^\nabla $ it follows:
\[
\frac{\partial H^\nabla }{\partial y_{ab,c}}
=\sum _{i\leq j}\frac{\partial ^2L^\nabla }
{\partial y_{ab,c}\partial y_{ij,k}}y_{ij,k},
\]
and the formula \eqref{matrix1} above. Moreover,
we claim that the functions $p_w^{uv}$ depend linearly
on the variables $y_{ab,c}$. In fact, as a calculation shows,
\begin{align*}
\frac{\partial p_w^{uv}}{\partial y_{ab,c}}
& =\frac{\partial ^2L^\nabla }
{\partial y_{ab,c}\partial y_{uv,w}}\\
& =\rho y^{jr}\frac{\partial ^2}
{\partial y_{ab,c}\partial y_{uv,w}}
\left(
G_{ji}^lG_{rl}^i-G_{li}^lG_{rj}^i
\right) \\
& =\tfrac{1}{(1+\delta_{ab})(1+\delta_{uv})}
\rho
\left\{
y^{bw}
\left(
y^{au}y^{cv}+y^{av}y^{cu}
\right)
+y^{aw}
\left(
y^{bu}y^{cv}+y^{bv}
y^{cu}
\right)
\right. \\
& -y^{ab}\left(  y^{cu}y^{vw}+y^{cv}y^{uw}
\right)
-y^{uv}\left(
y^{aw}
y^{bc}+y^{ac}y^{bw}
\right) \\
& \left.
-\left(
y^{ua}y^{vb}+y^{ub}y^{va}
\right)
y^{wc}+2y^{ab}
y^{uv}y^{wc}
\right\} .
\end{align*}
Therefore, in order to prove that the matrix \eqref{matrix1}
is non-singular, it suffices to prove that the variables
$y_{ab,c}$ can be written in terms of the functions $p_w^{uv}$.
To do this, we first compute
\begin{align*}
\sum _{u,v,w}
\tfrac{1+\delta_{uv}}{\rho}p_w^{uv}y_{ur}y_{vs}y_{wq}
& =y_{qr,s}+y_{qs,r}-y_{rs,q}\\
& -\tfrac{1}{2}\sum_{a,b}y^{ab}
\left(
y_{sq}y_{ab,r}+y_{rq}y_{ab,s}
\right) \\
& +\sum_{a,b}y^{ab}y_{rs}
\left(
y_{ab,q}-y_{qa,b}
\right) .
\end{align*}
Evaluating the previous formula at $g_{x_0}$, by using adapted
coordinates (i.e., $y_{ij}(g_{x_0})=\varepsilon _i\delta _{ij}$,
$\varepsilon _i=\pm 1$), and letting $\Upsilon _{rsq}(j_{x_0}^1g)
=\tfrac{1+\delta _{rs}}{\rho }p_q^{rs}(j_{x_0}^1g)
\varepsilon _r\varepsilon _s\varepsilon _q$, it follows:
\begin{align*}
\Upsilon _{rsq}(j_{x_0}^1g)
& =y_{qr,s}(j_{x_0}^1g)+y_{qs,r}(j_{x_0}^1g)
-y_{rs,q}(j_{x_0}^1g)\\
& -\tfrac{1}{2}\sum\nolimits_a\varepsilon _a\varepsilon _q
\left(
\delta _{sq}y_{aa,r}(j_{x_0}^1g)+\delta_{rq}y_{aa,s}(j_{x_0}^1g)
\right)
\\
& +\sum\nolimits_a\varepsilon_a\varepsilon_r\delta_{rs}
\left(
y_{aa,q}(j_{x_0}^1g)-y_{qa,a}(j_{x_0}^1g)
\right) .
\end{align*}
If $q\neq r\neq s\neq q$, then
$\Upsilon _{rsq}(j_{x_0}^1g)=y_{qr,s}
(j_{x_0}^1g)+y_{qs,r}(j_{x_0}^1g)-y_{rs,q}(j_{x_0}^1g) $.
Hence
\begin{equation}
y_{qr,s}(j_{x_0}^1g)=\tfrac{1}{2}
\left(
\Upsilon _{rsq}(j_{x_0}^1g)+\Upsilon _{qsr}(j_{x_0}^1g)
\right) .
\label{y_rsq}
\end{equation}
If $q=r$, $r\neq s$, then
\begin{equation}
\Upsilon _{rsr}(j_{x_0}^1g)
=y_{rr,s}(j_{x_0}^1g)-\tfrac{1}{2}
\sum\nolimits_a\varepsilon_a
\varepsilon_ry_{aa,s}(j_{x_0}^1g).
\label{Upsilon_rsr}
\end{equation}
If $r=s$, $q\neq r$, then
\begin{align}
\Upsilon _{rrq}(j_{x_0}^1g)
& =2y_{qr,r}(j_{x_0}^1g)-y_{rr,q}
(j_{x_0}^1g)
\label{Upsilon_rrq}\\
& +\sum\nolimits_a\varepsilon_a\varepsilon_r
\left(
y_{aa,q}(j_{x_0}^1g)-y_{qa,a}(j_{x_0}^1g)
\right) .
\nonumber
\end{align}
The formula \eqref{Upsilon_rsr} can be rewritten as
\[
2\varepsilon _r\Upsilon _{rsr}(j_{x_0}^1g)
=\varepsilon _ry_{rr,s}(j_{x_0}^1g)
-\sum\nolimits_{a\neq r}\varepsilon _ay_{aa,s}(j_{x_0}
^1g).
\]
Summing up over the index $r$, $2
{\displaystyle\sum\nolimits_r}
\varepsilon_r\Upsilon _{rsr}(j_{x_0}^1g)=(2-n)
{\displaystyle\sum\nolimits_r}
\varepsilon_ry_{rr,s}(j_{x_0}^1g)$, and replacing this formula
into \eqref{Upsilon_rsr} it follows:
\[
\Upsilon _{rsr}(j_{x_0}^1g)
=y_{rr,s}(j_{x_0}^1g)-\tfrac{1}{2-n}
\varepsilon _r
{\displaystyle \sum\nolimits_a}
\varepsilon _a\Upsilon _{asa}(j_{x_0}^1g).
\]
Therefore
\begin{equation}
y_{rr,s}(j_{x_0}^1g)=\Upsilon _{rsr}(j_{x_0}^1g)
+\tfrac{\varepsilon _r}{2-n}
{\displaystyle \sum\nolimits_a}
\varepsilon _a\Upsilon _{asa}(j_{x_0}^1g).
\label{yrrs}
\end{equation}
Replacing \eqref{yrrs} into \eqref{Upsilon_rrq},
we eventually obtain
\begin{equation}
\sum\nolimits_a\varepsilon _ay_{qa,a}(j_{x_0}^1g)
=\tfrac{1}{n-2}
\sum\nolimits_a\varepsilon _a\Upsilon _{aaq}(j_{x_0}^1g)
-2\tfrac{n-1}{(n-2)^2}
{\displaystyle \sum\nolimits_a}
\varepsilon_a\Upsilon _{aqa}(j_{x_0}^1g),
\label{1}
\end{equation}
and replacing $y_{rr,q}(j_{x_0}^1g)$,
$\sum\nolimits_a\varepsilon _ay_{aa,q}(j_{x_0}^1g)$,
and $\sum\nolimits_a\varepsilon_ay_{qa,a}(j_{x_0}^1g)$
into \eqref{Upsilon_rrq} it follows:
\begin{align*}
\Upsilon _{rrq}(j_{x_0}^1g)
& =2y_{qr,r}(j_{x_0}^1g)
-\Upsilon _{rqr}(j_{x_0}^1g)
+\tfrac{n\varepsilon_r}
{\left(
n-2
\right) ^2}
{\displaystyle \sum\nolimits_a}
\varepsilon _a\Upsilon _{aqa}(j_{x_0}^1g)\\
& -\tfrac{\varepsilon _r}{n-2}
\sum\nolimits_a\varepsilon _a
\Upsilon _{aaq}(j_{x_0}^1g).
\end{align*}
Hence
\begin{align}
y_{qr,r}(j_{x_0}^1g)
& =\tfrac{1}{2}\tfrac{n-1}{n-2}
\Upsilon _{rrq}(j_{x_0}^1g)
+\tfrac{1}{2}
\left(
1-\tfrac{n}{
\left(
n-2
\right)
^2}\right)
\Upsilon _{rqr}(j_{x_0}^1g)
\label{yqrr}\\
& -\tfrac{n\varepsilon _r}{2
\left(
n-2
\right) ^2}
{\displaystyle \sum\nolimits_{a\neq r}}
\varepsilon _a\Upsilon _{aqa}(j_{x_0}^1g)
\nonumber\\
& +\tfrac{\varepsilon_r}{2
\left(
n-2
\right) }
\sum\nolimits_{a\neq r}\varepsilon_a
\Upsilon _{aaq}(j_{x_0}^1g).
\nonumber
\end{align}
If $q=r=s$, then $\Upsilon _{rrr}(j_{x_0}^1g)
=-\sum\nolimits_{a\neq r}\varepsilon_a
\varepsilon_ry_{ra,a}(j_{x_0}^1g)$. From \eqref{1} we obtain
$\sum\nolimits_{a\neq r}\varepsilon_ay_{ra,a}(j_{x_0}^1g)$
and then
\begin{align*}
\sum\nolimits_{a\neq r}\varepsilon_ay_{ra,a}(j_{x_0}^1g)
& =-\varepsilon_ry_{rr,r}(j_{x_0}^1g)
+\tfrac{1}{n-2}\sum\nolimits_a\varepsilon _a
\Upsilon _{aar}(j_{x_0}^1g)\\
& -2\tfrac{n-1}{(n-2)^2}
{\displaystyle \sum\nolimits_a}
\varepsilon_a\Upsilon _{ara}(j_{x_0}^1g)
\end{align*}
and replacing it into the previous equation,
\[
\Upsilon _{rrr}(j_{x_0}^1g)=y_{rr,r}(j_{x_0}^1g)
-\tfrac{\varepsilon _r}{n-2}\sum\nolimits_a
\varepsilon _a\Upsilon _{aar}(j_{x_0}^1g)
+2\tfrac{(n-1)\varepsilon _r}{(n-2)^2}
{\displaystyle \sum\nolimits_a}
\varepsilon_a\Upsilon _{ara}(j_{x_0}^1g).
\]
Hence
\begin{align}
y_{rr,r}(j_{x_0}^1g)
& =\Upsilon _{rrr}(j_{x_0}^1g)
+\varepsilon _r
\tfrac{1}{n-2}\sum\nolimits_a\varepsilon _a
\Upsilon _{aar}(j_{x_0}^1g)
\label{yrrr}\\
&
-2\varepsilon_r\tfrac{n-1}{(n-2)^2}
{\displaystyle\sum\nolimits_a}
\varepsilon_a\Upsilon _{ara}(j_{x_0}^1g).
\nonumber
\end{align}
The formulas \eqref{y_rsq}, \eqref{yrrs}, \eqref{yqrr},
and \eqref{yrrr} end the proof.
\end{proof}
\subsection{Hamilton-Cartan equations}
The Poincar\'e-Cartan form for the density $L^\nabla \mathbf{v}$
is the $n$-form on $J^1\mathfrak{M}$ given by
\[
\Theta_{L^\nabla \mathbf{v}}
=\sum_{i\leq j}(-1)^{k-1}p_k^{ij}dy_{ij}\wedge \mathbf{v}_k
-H^\nabla \mathbf{v},
\]
the momenta $p_k^{ij}$ and the Hamiltonian function $H^\nabla $
being defined as in \eqref{p's_H}, and the Hamilton-Cartan equations
can geometrically be written as
\begin{equation}
\label{HC}
\left(
j^1g\right) ^\ast
\left(
i_Yd\Theta _{L^\nabla \mathbf{v}}
\right) =0,
\end{equation}
for every $p^1$-vertical vector field $Y\in J^1\mathfrak{M}$,
which are known to be equivalent to Euler-Lagrange equations,
 where $p^1\colon J^1\mathfrak{M}\to M$ is the natural projection.

According to Proposition \ref{proposition1}, $(x^i,y_{jk},p_w^{uv})$,
$j\leq k$, $u\leq v$, is a coordinate system on $J^1\mathfrak{M}$.
Letting $Y=\partial /\partial y_{ab}$
and $Y=\partial /\partial p_w^{uv}$
in \eqref{HC}, it follows respectively:
\begin{align*}
\sum _k\frac{\partial
\left(
p_k^{ab}\circ  j^1g
\right) }
{\partial x^k}
& =-\frac{\partial H^\nabla }{\partial y_{ab}}\circ j^1g,\\
\frac{\partial
\left(
y_{uv}\circ
j^1g\right) }
{\partial x^w}
& =\frac{\partial H^\nabla }{\partial p_w^{uv}}\circ j^1g,
\end{align*}
which are the Hamilton-Cartan equations in the canonical formalism.
\subsection{Covariant Hamiltonian}
An Ehresmann (or non-linear) connection on a fibred manifold
$p\colon E\to  M$ is a differential $1$-form $\gamma $ on $E$
taking values in the vertical sub-bundle $V(p)$ such that
$\gamma (X)=X$ for every $X\in V(p)$, e.g., see \cite{MSarda},
\cite{MSh}, \cite{Sarda}. Given $\gamma $, one has
$T(E)=V(p)\oplus \ker \gamma $, $\ker\gamma $
being the horizontal sub-bundle attached to $\gamma $.

According to \cite{MSh}, the covariant Hamiltonian $\mathcal{H}^\gamma $
associated to a Lagrangian density $\Lambda$ on $J^1E$ with respect to
$\gamma $ is the Lagrangian density defined by setting
$\mathcal{H}^\gamma
=\left(
(p_0^1)^\ast \gamma -\theta
\right)
\wedge  \omega _\Lambda -\Lambda $, where $p^1\colon  J^1E\to  M$,
$p_0^1\colon J^1E\to  J^0E=E$ are the natural projections,
and $\omega_\Lambda $ is the Legendre form attached to $\Lambda $,
i.e., the $V^\ast (p)$-valued $p^1$-horizontal $(n-1)$-form on $J^1E$
given by
\[
\omega _\Lambda
=(-1)^{i-1}\frac{\partial L}{\partial y_i^\alpha }
dx^1\wedge \cdots \wedge \widehat{dx^i}\wedge \cdots \wedge dx^n
\otimes dy^\alpha ,
\quad
\Lambda =L\mathbf{v},
\]
and $\theta =\theta ^\alpha \otimes \partial /\partial y^\alpha $,
$\theta ^\alpha =dy^\alpha -y_i^\alpha dx^i$, is the $V(p)$-valued
contact $1$-form on $J^1E$. Locally, $\mathcal{H}^\gamma
=\left(
\left(
\gamma _i^\alpha +y_i^\alpha
\right)
\frac{\partial L}{\partial y_i^\alpha }-L
\right)
\mathbf{v}$.

Let $\pi\colon  F(M)\to  M$ be the bundle of linear frames and let
$q\colon  F(M)\to $ $\mathfrak{M}$ be the projection given by
$q(X_{1},\dotsc,X_{n})=g_x=\varepsilon_{h}w^h\otimes w^h$, where
$(w^1,\dotsc,w^n)$ is the dual coframe of $(X_{1},\dotsc,X_n)
\in F_x(M)$, i.e., $g_x$ is the metric for which $(X_1,\dotsc,X_n)$
is a $g_x$-orthonormal basis and $\varepsilon_h=1$ for $1\leq h\leq n^+$,
$\varepsilon_{h}=-1$ for $n^++1\leq h\leq n$. The projection $q$
is a principal $G$-bundle with $G=O(n^+,n^-)$. Given a symmetric linear
connection $\Gamma $ with associated covariant derivative $\nabla $,
and a tangent vector $X\in T_xM$, for every $u\in \pi^{-1}(x)$
there exists a unique $\Gamma $-horizontal tangent vector
$X_u^{h_\Gamma }\in T_u(FM)$ such that, $\pi _\ast X_u^{h_\Gamma }=X$.
Given a metric $g_x\in q^{-1}(x)$, let $u\in \pi^{-1}(x)$
be a linear frame such that $q(u)=g_x$. The projection
$q_\ast (X_{u}^{h_{\Gamma _x}})$ does not depend on the linear frame
$u$ chosen over $g_x$; we refer the reader to \cite[Lemma 3.3]{MR}
for a proof of this fact. In this way a section
$\sigma^\nabla \colon  p^\ast TM\to T\mathfrak{M}$
of the projection $p_\ast \colon T\mathfrak{M}\to  p^\ast TM$
is defined by setting $\sigma ^\nabla (g_x,X)
=q_\ast (X_{u}^{h_{\Gamma _x}})$. The retract
$\gamma ^\nabla \colon  T\mathfrak{M}\to  V(p)$ associated
to $\sigma ^\nabla $, namely,
$\gamma ^\nabla (Y)=Y-\sigma ^\nabla (p_\ast Y)$,
$\forall Y\in T_{g_x}\mathfrak{M}$, determines an Ehresmann
connection on the bundle of metrics and the Lagrangian density
$\Lambda^\nabla =L^\nabla \mathbf{v}$ admits a ``canonical''
covariant Hamiltonian $\mathcal{H}^{\gamma ^\nabla }$. Locally,
\[
\gamma ^\nabla
\left(
g_x,\partial /\partial x^j
\right)
=-\sum_{k\leq l}
\left\{
\Gamma _{jk}^a(x)y_{al
}\left(
g_x\right)
+\Gamma _{jl}
^a(x)y_{ak}
\left(
g_x
\right)
\right\}
\!\left(  \!\partial /\partial
y_{kl}\!\right)  _{g_x}.
\]
Hence, $\gamma _{kl,j}
=-\left( \Gamma _{jk}^ay_{al}+\Gamma _{jl}^ay_{ak}
\right) $, and
\[
\mathcal{H}^{\gamma ^\nabla }
=\left(
\sum_{k\leq l}
\left(
y_{kl,j}-\left(
\Gamma _{jk}^ay_{al}+\Gamma _{jl}^ay_{ak}
\right)
\right)
\frac{\partial L^\nabla }{\partial y_{kl,j}}-L^\nabla
\right)
\mathbf{v}.
\]
From a direct computation the following result is deduced:

If $\mathcal{H}^{\gamma ^\nabla }=H^{\gamma ^\nabla }\mathbf{v}$,
then
\[
H^{\gamma ^\nabla }(j_x^1g)
=L^\nabla (j_x^1g)-2\rho (g_x)s^{g,\nabla }(x),
\quad
\forall  j_x^1g\in J_x^1\mathfrak{M},
\]
where $s^{g,\nabla}$ is the scalar curvature of the symmetric
linear connection $\nabla $ with respect to the metric $g$,
namely
\[
s^{g,\nabla }=g^{jk}
\left\{
\tfrac{\partial\Gamma _{jk}^i}{\partial x^i}
-\tfrac{\partial\Gamma _{ik}^i}{\partial x^j}
+\Gamma _{jk}^l\Gamma _{il}^i-\Gamma _{ik}^l\Gamma _{jl}^i
\right\} .
\]
The Hamilton-Cartan equations for a covariant Hamiltonian
$H^{\gamma}$ attached to a connection $\gamma $ are
\begin{align*}
\sum _k\frac{\partial
\left(
p_k^{ab}\circ  j^1g
\right) }
{\partial x^k}-\sum_{u\leq v}
\left(
\frac{\partial\gamma _{uv,w}}{\partial y_{ab}}\circ g
\right)
\left(
p_w^{uv}\circ  j^1g
\right)
& =-\frac{\partial
H^{\gamma}}{\partial y_{ab}}\circ  j^1g,\\
\frac{\partial
\left(
y_{uv}\circ j^1g
\right) }{
\partial x^w}
+\gamma _{uv,w}\circ  j^1g & =\frac{\partial H^{\gamma}} {\partial
p_w^{uv}}\circ j^1g,
\end{align*}
(for example see \cite{CM}). Note that for $\gamma = 0$ (that is,
the trivial connection induced by the coordinate system) these
equations coincide with the local expression of the
Hamilton-Cartan equations for $H^\nabla $ given in \S5.2.
\section{Conclusions}
We have defined a first-order Lagrangian $L^\nabla $ on the bundle
of metrics which is variationally equivalent to the second-order
classical Einstein-Hilbert Lagrangian.

This Lagrangian depends on an auxiliary symmetric linear connection,
but this dependence is covariant under the action of the group
of diffeomorphisms.

We have also proved that the variational problem defined by $L^\nabla $
is regular and its Hamiltonian formulation has been studied, including
the covariant Hamiltonian attached to $\nabla $.

Moreover, we shoud finally mention the completely different behaviour
of $L^\nabla $ with respect to the Palatini Lagrangian.

Let $q\colon \mathfrak{C}\to  M$ be the bundle of symmetric linear
connections on $M$. The Palatini variational principle consists
in coupling a metric $g$ and a symmetric linear connection $\nabla $
as independent fields, thus defining a first-order Lagrangian density
$L_P\mathbf{v}$ on the product bundle $\mathfrak{M}\times _M\mathfrak{C}$
as follows:
\[
\left(
L_P\mathbf{v}
\right)
\left(
g_x,j_x^1\nabla
\right)
=s^{g,\nabla }(x)(\mathbf{v}_g)_x,
\]
and varying $g$ and $\nabla $ independently. The Palatini method
can also be applied to other different settings; e.g., see \cite{KL},
\cite{DP}, but below we confine ourselves to consider the classical
setting for the Palatini method. As is known, the Euler-Lagrange
equations of $L_{P}$ are the vanishing of the Ricci tensor of $g$
(Einstein's in the vacuum) and the condition $\nabla=\nabla^g$
expressing that $\nabla $ is the Levi-Civita connection of the metric.

In our case, we can similarly define a first-order Lagrangian
$\mathfrak{M}\times_M\mathfrak{C}$ by setting $L(j^1g,j^1\nabla )
=L_{HE}(j^2g)+
c\left(
(\operatorname*{alt}\nolimits_{23}(\nabla ^gT^{g,\nabla })^\sharp
\right)
\left(
\rho\circ g
\right) $.
Assuming $M$ is compact, then the action associated with $L$
is given as follows:
$\mathcal{S}(g,\nabla )=\in t_ML^\nabla (j^1g,j^1\nabla )\mathbf{v}$,
and by considering 1) an arbitrary $1$-parameter variation $g_t$ of $g$
and 2) the $1$-parameter variation $\nabla _t=\nabla +tA$ attached
to $A\in \Gamma (S^2T^\ast M\otimes TM)$ of $\nabla $, we obtain,
1) Einstein's equation and 2) $0=\in t_Mc
\left(
\operatorname*{alt}\nolimits_{23}(\nabla ^gA)^\sharp
\right)
\mathbf{v}_g$, $\forall A\in \Gamma (S^2T^\ast M\otimes TM)$,
which leads us to a contradiction.

\end{document}